\documentclass[a4paper,12pt]{amsart}

\usepackage[usenames,dvipsnames]{xcolor}

\usepackage{tikz}
\usetikzlibrary{shapes,backgrounds}

\usepackage{blindtext}

\usepackage[utf8x]{inputenc}

\usepackage[a4paper,top=35mm,bottom=35mm,inner=28mm,outer=28mm]{geometry}
\usepackage[colorlinks,linkcolor=black,citecolor=black,urlcolor=black,hypertexnames=true]{hyperref}
\newcommand{\realset}[1]{{Z_\RR(#1)}}
\newcommand{\zeroset}[1]{{Z(#1)}}

\newcommand\restr[2]{{
  \left.\kern-\nulldelimiterspace 
  #1 
  \vphantom{\big|} 
  \right|_{#2} 
  }}

\usepackage[T1]{fontenc}

\usepackage{graphicx}

\DeclareMathOperator\mult{mult}

\newcommand{\zar}[1]{\rm{Zar}(#1)}

\newcommand{\boundary}{\partial}

\newcommand{\bc}{\begin{center}}
\newcommand{\ec}{\end{center}}

\newcommand{\scrE}{\mathscr{E}}

\usepackage{bbm}

\usepackage{mathtools,amsthm}

\def\NN{{\mathbbm{N}}}
\def\ZZ{{\mathbbm{Z}}}
\def\QQ{{\mathbbm{Q}}}
\def\RR{{\mathbbm{R}}}
\def\CC{{\mathbbm{C}}}

\def\PP{{\mathbbm{P}}}

\usepackage{fancyvrb}

\usepackage{color}
\definecolor{mycolor}{rgb}{0.28,0.69,0.7}
\definecolor{mycolor1}{rgb}{0.7,0.75,0.71}
\definecolor{mycolor2}{rgb}{0.58,1,1}

\usepackage{varwidth,url}     
\usepackage{algorithmicx,caption}
\usepackage{algorithm}
\usepackage[noend]{algpseudocode}
\makeatletter
\def\BState{\State\hskip-\ALG@thistlm}
\makeatother

\def\ss{{\mathbbm{S}}}
\def\Id{{\mathbbm{I}}}

\usepackage{amsmath,caption}
\usepackage{amssymb}
\usepackage{amsfonts,comment}
\usepackage{graphicx,epstopdf,multicol}
\usepackage{mathrsfs}
\usepackage{subfigure}

\newtheorem{theorem}{Theorem}
   
\newtheorem{proposition}[theorem]{Proposition}   
\newtheorem{lemma}[theorem]{Lemma}   
\newtheorem{corollary}[theorem]{Corollary}   

\newtheorem{remark}[theorem]{Remark}

\theoremstyle{remark}
\newtheorem{example}[theorem]{Example}   
\newtheorem{exper}[theorem]{Test}

\def\spec{\mathscr{S}}
\def\cc{\mathcal{C}}
\def\ch{\mathrm{ch}}

\def\deg{{\rm deg}}

\def\mymid{{\,\,:\,\,}}

\newcommand\Secgamma{\Gamma'}
\newcommand\Seclambda{\Lambda'}

\newcommand{\new}[1]{{#1}}

\title{Symbolic computation in hyperbolic programming}

\author{Simone Naldi}
\author{Daniel Plaumann}
\address{Technische Universit\"at Dortmund,
  Fakult\"at f\"ur Mathematik, Vogelpothsweg 87, 44227 Dortmund}
\email{simone.naldi@tu-dortmund.de}
\email{daniel.plaumann@tu-dortmund.de}

\date{\today}

\begin{document}

\maketitle

\begin{abstract}
\noindent
Hyperbolic programming is the
problem of computing the infimum of a linear function when restricted
to the hyperbolicity cone of a hyperbolic polynomial, a generalization of
semidefinite programming. We propose an approach based on symbolic
computation, relying on the multiplicity structure of the algebraic
boundary of the cone, without the assumption of determinantal
representability. This allows us to design exact algorithms able to
certify the multiplicity of the solution and the optimal value of the
linear function.
\end{abstract}

\section*{Introduction}
\label{sec:intro}
Semidefinite programming (SDP) constitutes a popular class of convex
optimization problems for which approximate solutions can be
computed through a variety of numerical algorithms, the most
efficient of which are based on primal-dual interior point methods.
On the other hand, exact algorithms for general semidefinite
programs have been developed only much more recently in the work of
Henrion, Safey El Din and the first author in \cite{HNS2015c}. The
optimizer in an SDP problem corresponds to a positive
semidefinite real symmetric matrix. While symbolic algorithms
obviously have a much higher complexity than numerical ones, finding
exact solutions has many benefits, especially regarding certification
of the solution. For instance, the rank of the optimizer, which is
often a meaningful quantity, can be determined exactly by the algorithm
in \cite{HNS2015c}, and one can even optimize over all feasible points
of bounded rank, which is a non-convex optimization problem \cite{naldiIssac2016}.

In this paper, we consider analogous algorithmic questions in the more
general framework of \emph{hyperbolic programming}. We briefly
summarize the underlying notions. A real homogeneous polynomial $f$ in
several variables $x=(x_1,\dots,x_n)$ is called \emph{hyperbolic} with
respect to a point $e\in\RR^{n}$ if $f(e)\neq 0$ and the polynomial
$f(te-a) \in \RR[t]$ has only real zeros for every $a \in \RR^{n}$. The
general determinant of symmetric matrices has this property with
respect to the unit matrix $e=\Id_d$, since $\det(t\Id_d-A)$ is the 
classical characteristic polynomial of the real symmetric matrix $A$.
Hyperbolic polynomials
can therefore be seen as generalized characteristic polynomials.
If $f$ is hyperbolic with respect to $e$, the hypersurface defined by $f$
bounds a convex cone containing $e$, the \emph{hyperbolicity cone},
generalizing the cone of positive semidefinite matrices in case of the
determinant. The zeros of $f(te-a)$ can be regarded as generalized
eigenvalues of $a \in \RR^{n}$, and the multiplicity of the root $t=0$ of
$f(te-a)$ as the corank of $a$. 

A \emph{hyperbolic program} is the
convex optimization problem of minimizing a linear function over the
hyperbolicity cone of a hyperbolic polynomial. \new{Such cones have non-empty
interior by construction (the interior will indeed contain the point $e$).
Denote now by $\ss_d(\RR)$ the set of $d \times d$ real symmetric matrices.
{\it Regular} semidefinite programs (in which the feasible set has non-empty interior) correspond
to the} case in which $f$ is the restriction of the determinant
\new{map $\det \colon \ss_d(\RR) \to \RR$} to a linear subspace \new{$V
\subset \ss_d(\RR)$ containing a positive definite matrix.
More precisely, for such $V$, the polynomial $f = \restr{\det}{V}$ is hyperbolic,
and its hyperbolicity cone is the spectrahedron $V \cap \ss^+_d(\RR)
= \{M \in V \mymid M \succeq 0\}$. If $V$
does not contain positive definite matrices, $V \cap \ss^+_d(\RR)$ is still
a spectrahedron, but $f = \restr{\det}{V}$ is not hyperbolic.

Moreover, not every hyperbolic polynomial can be represented in this way
(in fact, the set of representable polynomials is, in general, of strictly smaller dimension) which motivates the development of techniques that are independent of the
determinantal representability of $f$.}

Hyperbolic programming can be solved numerically with interior point methods much like SDP
\cite{deklerk,guler1997hyperbolic,nestnem}.
  \begin{figure}[h]
    \centering
\includegraphics[width=10cm]{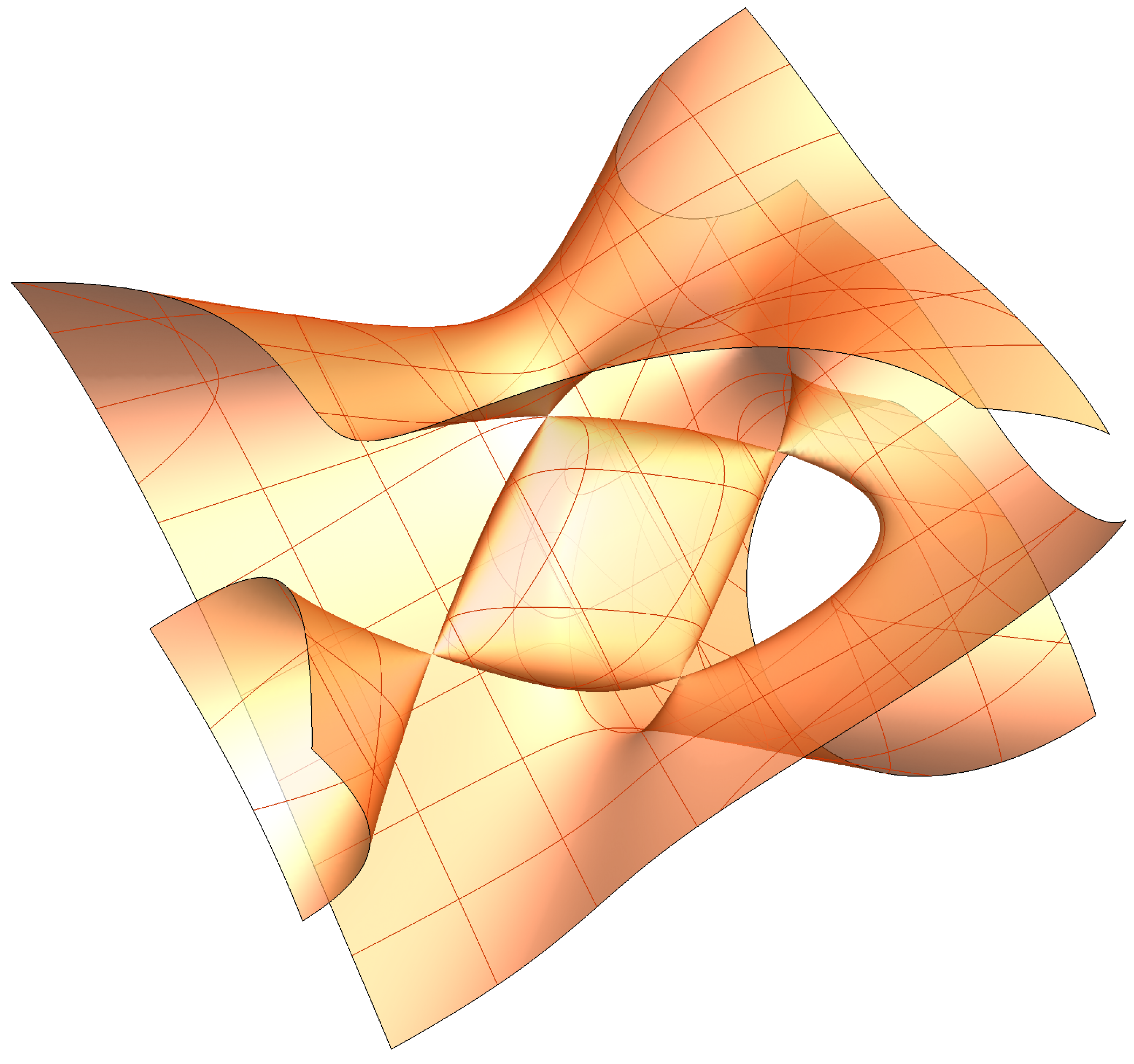}\\
    \caption{\small An affine section of a hyperbolic quartic surface with
      four nodes}\label{Fig:NodalQuartic}
  \end{figure}
One of the major challenges in hyperbolic programming, when compared to
SDP, is the lack of an explicit duality theory,
while SDP duality is always heavily exploited. The methods in
\cite{HNS2015c} rely on the good properties of
determinantal varieties, which provide an explicit non-singular
lifting of the variety of symmetric matrices of bounded rank in a
given subspace.  The same is not available for hyperbolic programming.
However, hyperbolicity of a real polynomial still imposes some
strong conditions on the structure of the real part of the singular
locus of the hypersurface.

Let us give an overview of our main results. Given a polynomial $f$, hyperbolic with respect
to $e \in \RR^n$, let $\Lambda_+$ be the hyperbolicity cone \new{of $f$} (Section \ref{sec:hypcones}) and let $\Gamma_m \subset \RR^n$
denote the set of points of multiplicity at least $m$ (see Section \ref{sec:multiplicity}). Furthermore, let \new{$L_e = \{x \in \RR^n \mymid e^Tx = 1\}$ be the affine space orthogonal
to the direction $e$ (containing $\frac{e}{\|e\|^2}$)
and write $\Lambda_+'=\Lambda_+\cap L_e$ and $\Gamma'_m=\Gamma_m\cap L_e$.} We show that if $m$ equals
the maximal multiplicity on $\new{\Seclambda_+}$, then
$\new{\Seclambda_+}$ contains one of the real connected components of $\new{\Secgamma_m}$,
proving that $\new{\Secgamma_m} \cap \Lambda_+$ is the union of some
components of $\new{\Secgamma_m}$ (Proposition \ref{prop-component}).
Thus a point of maximal multiplicity (analogous to the minimal rank
in SDP) can be found by sampling the connected components of $\new{\Secgamma_m}$.
Since this is an algebraic set (rather than just semialgebraic),
this reduces to a standard problem in computational real algebraic
geometry.
Furthermore, we show that the more general convex hyperbolic programming
problem over $\new{\Seclambda_+}$ is equivalent to computing local minimizers
over the sets $\new{\Secgamma_m}$ of the same linear function (Theorem \ref{theo:reduction}).
This can be carried out in practice using Lagrange multipliers, provided
that the corresponding set of critical points has complex dimension $0$.
We use these results to design an exact algorithm for hyperbolic
programming. Applying this to explicit examples yields interesting
results that are discussed in the final part.

The structure of this paper is as follows. In Section
\ref{sec:hypcones}, we summarize
standard definitions and results about hyperbolicity cones. In Section \ref{sec:multiplicity},
we describe the multiplicity structure of the algebraic boundary of $\Lambda_+$ and
prove our main result on the maximal multiplicity. This is used to certify feasible
multiplicities in the case where $\Lambda_+$ is the $d$-elliptope.
In Section \ref{sec:HP}, we formalize the relationship between solutions to hyperbolic
programming problems and the multiplicity loci $\Gamma_m$. Our algorithm solving
hyperbolic programming in exact arithmetic is implemented in Maple; we
finally discuss the
results of our tests.

\medskip
{\textit{Acknowledgements.} Work on this project was partially
  supported through DFG grant PL 549/3-1 \textit{Convexity in Real
    Algebraic Geometry}.


\section{Preliminaries}
\label{sec:hypcones}

The following notation is used throughout. The ring of real
polynomials in $x=(x_1, \ldots, x_n)$ is denoted by $\RR[x]$, with its
natural grading $\RR[x]=\bigoplus_d \RR[x]_d$.  The complex algebraic
set defined by polynomials \new{$f = (f_1, \ldots, f_s) \in \RR[x]^s$}
is denoted by
$\zeroset{f} = \{x \in \CC^{n} \mymid \forall\,i \,\, f_i(x) = 0\}$,
and its subset of real points by $\realset{f} = \zeroset{f} \cap \RR^{n}$. The
closure of a set $S \subset \RR^{n}$ in the Euclidean topology
(resp.~Zariski topology) of $\RR^{n}$ is $\overline{S}$
(resp.~$\zar{S}$). Its Euclidean boundary is $\partial S$.

\medskip Recall from the introduction that a homogeneous polynomial
$f\in\RR[x]$ of degree $d$ is called \emph{hyperbolic} with respect to a point
$e\in\RR^{n}$ if $f(e)\neq 0$ and \new{$f(te-a)\in\RR[t]$} has only real roots
for every $a\in\RR^{n}$. Up to rescaling $f$, we may suppose that $f(e) =
1$, and we often say that $f$ is just {\it hyperbolic}, without
specifying the direction $e$.

The polynomial
\[
t \mapsto \ch_a(t;f,e) = f(t e-a)
\]
is called the {\it characteristic polynomial} of $a\in\RR^{n}$ (with respect to $f$ and $e$). This will be
denoted simply by $\ch_a(t)$ when $f,e$ are understood from the context. For
$a \in \RR^{n}$, the ordered roots
$
\lambda_1(a) \leq \lambda_2(a) \leq \cdots \leq \lambda_d(a) 
$
of $\ch_a$ are called the {\it eigenvalues} of $a$. The set
\[
\Lambda_{++}(f,e) = \{a \in \RR^{n} \mymid \lambda_1(a) > 0\}
\]
is called the {\it open hyperbolicity cone} of $f$ with respect to $e$. This is an open, convex cone containing $e$
\cite{garding1959inequality}. \new{Note that there may exist several hyperbolicity cones associated to a given
hyperbolic polynomial, and bounds on the number of such cones have been recently computed in \cite{thorstenthorsten};
however, when $f$ is irreducible, there is only one hyperbolicity cone, up to changing
$e$ with $-e$ (see \cite{kummer})}.
We will simply denote $\Lambda_{++}(f,e)$ by $\Lambda_{++}$ when $f,e$ are fixed. It is independent of the choice of $e$
within $\Lambda_{++}$, by the following basic result.

\begin{theorem}[G{\aa}rding \cite{garding1959inequality}, Renegar \cite{renegar2006hyperbolic}]
\label{theo:hypconeGarding}
The cone $\Lambda_{++}$ is convex, open and coincides with the connected component of $\RR^{n} \setminus \realset{f}$
containing $e$. Moreover, $f$ is hyperbolic with respect to $e'$ for all $e' \in \Lambda_{++}(f,e)$, and $\Lambda_{++}(f,e)
=\Lambda_{++}(f,e')$.
\end{theorem}
The closure of $\Lambda_{++}$ in the Euclidean topology is called the
\emph{hyperbolicity cone}. It is denoted by $\Lambda_+$ and equals 
\[
\Lambda_+ = \overline{\Lambda_{++}} = \{x \in \RR^{n} \mymid \lambda_1(x) \geq 0\}.
\]

\new{As mentioned in the introduction, we work over the affine section 
\[
\Seclambda_+=\Lambda_+\cap L_e,\qquad\text{where }
L_e = \{x \in \RR^n \mymid e^Tx = 1\}.
\]}
Note that the relative interior of $\new{\Seclambda_+}$ in $L_e$ is not empty, since, for instance, it contains $\frac{e}{\|e\|^2}$.
We call the affine section $\new{\Seclambda_+}$ the {\it standard section} since it preserves the multiplicity structure of the cone
$\Lambda_+$ (see Remark \ref{remarknice}).}

Since linear (LP) and semidefinite programming (SDP) are special instances of hyperbolic programming,
we recall the description of their feasible cones in this setting. For LP, the polynomial splits into real linear
factors, that is $f = \ell_1 \cdots \ell_d$, with $\ell_i \in \RR[x]_1$. For all $e \in \RR^{n}$ with $\ell_i(e)>0,
i=1, \ldots, d$,  one has that
\[
\Lambda_{+}\left(f, e\right)=\left\{ x \in \RR^{n} \mymid \forall\,i \,\,\, \ell_i(x) \geq 0\right\}
\]
is a polyhedron.
For SDP, the hyperbolicity cone of $f = \det(x_1A_1+\cdots+x_nA_n)$ in direction
$e$, with $A_i$ real symmetric, and $e_1A_1+\cdots+e_nA_n \succ 0$,
equals the cone of positive semidefinite matrices in the subspace
spanned by
$A_1,\dots,A_n$, that is, a spectrahedral cone.

\section{Multiplicities}
\label{sec:multiplicity}

Let $a \in \RR^{n}$, and $f \in \RR[x]$ a hyperbolic polynomial
(with respect to a fixed point $e\in\RR^{n}$) of degree $d$. The
{\it multiplicity} of $a$ is the multiplicity of $t = 0$ as an
eigenvalue of $a$, hence as a root of the characteristic polynomial
$\ch_a(t)=f(te-a)$, and is denoted by $\mult({a})$. We consider the
semialgebraic sets
\begin{align*}
\partial^{m}\Lambda_+ & = \{a \in \Lambda_+ \mymid \mult({a}) = m\}, \\
\partial^{\geq m}\Lambda_+ & = \partial^{m}\Lambda_+ \cup \partial^{m+1}\Lambda_+ \cup \cdots \cup \partial^{d}\Lambda_+,
\end{align*}
for $0 \leq m \leq d$, first considered in \cite{renegar2006hyperbolic}. \new{By \cite[Prop.22]{renegar2006hyperbolic},
the function $a \mapsto \mult(a)$ is independent of the hyperbolic direction $e$ chosen within $\Lambda_+(f,e)$.}
The following lemmas give basic results on the
multiplicity structure of the boundary of $\Lambda_+$ and on its algebraic closure. These will allow us to define
explicit equations for algebraic relaxations of $\Lambda_+$.

\begin{lemma}
\label{lemmaclosure}
For $0 \leq m \leq d$, the set $\partial^{\geq m}\Lambda_+$ is closed and $\overline{\partial^{m}\Lambda_+} \subset \partial^{\geq m}\Lambda_+$.
\end{lemma}
\begin{proof}
The inclusion ${\partial^{m}\Lambda_+} \subset \partial^{\geq
  m}\Lambda_+$ holds by definition. Hence we have to show that
$\partial^{\geq m}\Lambda_+$ is closed, which comes simply from the continuity of eigenvalues $a \mapsto \lambda_i(a)$.
Indeed, if $(a_{\epsilon})_{\epsilon > 0} \subset \partial^{\geq m}\Lambda_+$, then for all $\epsilon > 0$
\[
\lambda_1(a_\epsilon) = \cdots = \lambda_m(a_\epsilon) = 0 \leq \lambda_{m+1}(a_\epsilon) \leq  \cdots \leq \lambda_{d}(a_\epsilon).
\]
Thus if $a_\epsilon \xrightarrow[]{\epsilon \rightarrow 0^+} a$, then $\lambda_1(a) = \cdots \lambda_m(a) = 0 \leq \lambda_{m+1}(a)
\leq  \cdots \leq \lambda_{d}(a)$, hence $\mult({a}) \geq m$ as claimed.
\end{proof}

Note that, typically, neither $\partial^{m}\Lambda_+$ nor
$\partial^{\geq m}\Lambda_+$ are Zariski closed sets, since they are
semialgebraic rather than algebraic. Often, in order to develop
algebraic techniques, it is desirable to work with real algebraic
sets. We therefore define the following:
\begin{equation*}
\Gamma_m = \left\{a \in \RR^{n} \mymid \mult({a}) \geq m \right\}, \qquad 0 \leq m \leq d.
\end{equation*}
The sets $\Gamma_m$ define a nested collection in $\RR^{n}$:
\[
\Gamma_d \subset \Gamma_{d-1} \subset \cdots \subset \Gamma_1 = \{a \in \RR^{n} \mymid f(a)=0 \} = \realset{f} \subset \Gamma_0 = \RR^{n}.
\]
By Lemma \ref{lemmaclosure}, we have $\overline{\partial^m\Lambda_{+}} \subset \Gamma_m \cap \Lambda_{+}$ (indeed, $\Gamma_m \cap \Lambda_{+} =
\partial^{\geq m}\Lambda_+$).

\begin{lemma}
\label{lemma:eqLambda}
For any $0 \leq m \leq d$, the set $\Gamma_m$ is real algebraic and
satisfies
\[
\zar{\partial^m\Lambda_+}=\zar{\overline{\partial^m\Lambda_+}} \subset
\Gamma_m.
\]
\end{lemma}
\begin{proof}
The equality $\zar{S}=\zar{\overline{S}}$ always holds, since the Zariski topology is coarser than the Euclidean topology.
The inclusion $\zar{\overline{\partial^m\Lambda_+}}
\subset \zar{\Gamma_m}$ follows from $\overline{\partial^m\Lambda_{+}} \subset \Gamma_m \cap \Lambda_{+}$, proved
in Lemma \ref{lemmaclosure}. Hence we only need to prove that $\Gamma_m$ is real algebraic. Writing
\[
\ch_x(t) = t^d + g_{1}(x) t^{d-1} + \cdots + g_{d-1}(x) t + g_d(x),
\]
we deduce that a point $a\in\RR^{n}$ lies in $\Gamma_m$ if and only
if $t^m$ divides $\ch_a(t)$. This is the case if and only if all coefficients $g_d = f(-x), g_{d-1}, \ldots, g_{d-m+1}\in\RR[x]$ vanish
at $a$.
\end{proof}

The defining equations for $\Gamma_m$ obtained in the proof can be
made more explicit. 
Let us consider the {\it modified} characteristic polynomial $\ch_{-a}=f(t e + a)$
at a point $a \in \RR^{n}$, with respect to $f$ and $e$.
Denote by $\sigma_i(y_1, \ldots, y_d) = \sum_{j_1 < \cdots
  < j_i} y_{j_1}\cdots y_{j_i}$ the $i$-th elementary symmetric polynomial on variables
$y_1, \ldots, y_d$, then
\[
\ch_{-a}= t^d + \sigma_1(\lambda(a))t^{d-1} + \cdots + \sigma_{d-1}(\lambda(a))t + f(a)
\]
where $\lambda(a) = (\lambda_1(a),\ldots,\lambda_d(a))$ are the
eigenvalues of $a$. Note that $f(x) = \sigma_d(\lambda(x))=\lambda_1(x)\cdots \lambda_d(x)$,
since $f(e)=1$.

\begin{corollary}
\label{eqGamma}
For $0 \leq m \leq d$, we have $\Gamma_m=\realset{f,\sigma_{d-1}(\lambda),\ldots,\sigma_{d-m+1}(\lambda)}$.
\end{corollary}
\begin{proof}
Indeed, in the proof of Lemma \ref{lemma:eqLambda} we have shown that
$\Gamma_m$ equals the real algebraic set $\realset{g_d, g_{d-1}, \allowbreak \ldots, \allowbreak g_{d-m+1}}$,
with $g_d:=f(-x)$ and $g_i(x)$ is the coefficient of $t^{d-i}$ in $\ch_x(t)$.
Therefore, $g_{i}(-x) = \sigma_{i}(\lambda(x))$
for $i=1,\ldots, d$ and hence
$$
\Gamma_m = \realset{f(-x) , \sigma_{d-1}\bigl(\lambda(-x)\bigr),\ldots,\sigma_{d-m+1}\bigl(\lambda(-x)\bigr)}.
$$
The claim follows from the homogeneity of $f,\sigma_{d-1}(\lambda),\ldots,\sigma_{d-m+1}(\lambda)$.
\end{proof}


\new{
\begin{remark}
\label{remarknice}
\emph{Suppose that $a \in \Lambda_+$ with $e^Ta \neq 0$. Then $x = \frac{a}{e^Ta} \in \new{\Seclambda_+}$ and
$\mult(x) = \mult(a)$. We deduce that if $\Lambda_+ \cap \{x \in \RR^n \mymid e^Tx=0\} = \{0\}$, the standard section
of a hyperbolicity cone is a \emph{base} for $\Lambda_+$ (as a convex cone; see \cite[Def. 8.3]{barvinok}) and has the same multiplicity
structure as $\Lambda_+ \setminus \{0\}$.}
\end{remark}
}

We are particularly interested in computing the \emph{maximum multiplicity} on the \new{standard section of the }cone $\Lambda_{+}$. For a hyperbolic polynomial $f \in \RR[x]_d$, we define the integer
\[
\max_{a \in \new{\Seclambda_+}} \mult({a}) = \max_{0 \leq t \leq d} \{t \mymid \new{\Secgamma_t} \cap\Lambda_+ \neq \emptyset\}.
\]
This is well defined since $\new{\Seclambda_+} \neq \emptyset$ and $0 \leq \mult({a}) \leq d$ for all $a \in \RR^{n}$.
We now show the maximum multiplicity $m$ is attained on an entire real connected component of $\new{\Secgamma_m}$.

\begin{proposition}
\label{prop-component}
Let $f \in \RR[\new{x}]_d$ be hyperbolic with respect to $e \in \RR^{n}$, and let $m=\max_{a \in \new{\Seclambda_+}} \mult({a})$. For every (real)
connected component $\cc$ of $\new{\Secgamma_m}$, with $\cc \cap \Lambda_+ \neq \emptyset$, we have
\begin{enumerate}
\item[(1)]
  $\cc \subset \Lambda_+$
\item[(2)]
  $\cc \cap \Gamma_{m+1} = \emptyset$.
\end{enumerate}
\end{proposition}
\begin{proof}
First, note that (1) implies (2), because, by definition, $m$ is the
maximal multiplicity of points in $\Lambda_+$.

To prove (1), let $\cc$ be a connected component of $\new{\Secgamma_m}$ intersecting $\Lambda_+$, and let $a \in \cc \cap \Lambda_+$.
Note that $\mult({a}) = m$ by maximality of $m$. Suppose that there is $b
\in \cc \setminus \Lambda_+$; hence $\mult({b}) \geq m$ and $\lambda_1(b) < 0$. Since $\cc$ is
connected,
there is a continuous path $\varphi \colon [0,1] \rightarrow \cc$ with $\varphi(0)=a$ and $\varphi(1)=b$.
For all $t \in [0,1]$, $\mult({\varphi(t)}) \geq m$, and there exists $t_0 \in [0,1]$ such that
\begin{itemize}
\item[(a)]
  $\varphi([0,t_0]) \subset \Lambda_+$ and
\item[(b)]
  $\exists\,\delta > 0$ such that, $\forall\,0<\epsilon<\delta$, $\varphi(t_0+\epsilon) \notin \Lambda_+$. 
\end{itemize}
Define $a_0=\varphi(t_0)$ and $a_\epsilon=\varphi(t_0+\epsilon),$ for $0<\epsilon <\delta$, so that $a_\epsilon
\xrightarrow[]{\epsilon \rightarrow 0^+} a_0$. Since $a_\epsilon \notin \Lambda_+$, then
$\lambda_1(a_\epsilon)<0$ for all $0 < \epsilon < \delta$. More precisely, since $\mult({{a_\epsilon}}) \geq m$, then
for all $\epsilon > 0$ there exists $i(\epsilon) \in \{1, \ldots, d\}$ such that
\begin{align*}
\lambda_1(a_\epsilon) \leq \cdots \leq \lambda_{i(\epsilon)}(a_\epsilon) & < 0 \\
\lambda_{i(\epsilon)+1}(a_\epsilon) = \cdots = \lambda_{i(\epsilon)+m}(a_\epsilon) & = 0 \\
0 \leq \lambda_{i(\epsilon)+m+1}(a_\epsilon) \leq \cdots  & \leq \lambda_d(a_\epsilon).
\end{align*}
Passing to the limit for $\epsilon \rightarrow 0^+$, by the continuity of the eigenvalues and since $a_0 \in \Lambda_+$,
we find that $\lambda_1(a_0) = \cdots = \lambda_{m+1}(a_0)=0$, that is $\mult({a_0}) \geq m+1$,
which contradicts the fact that $m$ is the maximum multiplicity.
\end{proof}

Proposition \ref{prop-component} provides us with a way of computing the largest multiplicity on the \new{standard section
$\new{\Seclambda_+}$ of the} hyperbolicity cone $\Lambda_+$, and of representing one point where this maximum value is attained.
More precisely, consider the {\it non-convex} optimization problem
\begin{equation}
\label{maxmult}
 \boxed{
\begin{aligned}
\max & \,\,\, \mult({x}) \\
 s.t. & \,\,\, x \in \new{\Seclambda_+(f,e)}
\end{aligned}
 }
\end{equation}
By the hyperbolicity of $f$, \new{the interior of the feasible set} $\Lambda_{+}$ is $\Lambda_{++}
\neq \emptyset$ \new{(for instance, $e \in \Lambda_{++}$)}, hence this problem is feasible \new{(and with
non-empty interior)}. Hence there always exists $a^* \in \new{\Seclambda_+}$ such that
$\mult({a^*}) = m := \max_{a \in \new{\Seclambda_+}} \mult({a})$. By Proposition \ref{prop-component}, the whole
connected component $\cc^*$ of $\new{\Secgamma_m}$ containing $a^*$, is included in $\Lambda_+$.

\new{First, by Remark \ref{remarknice}, under the assumption that $\Lambda_+$ intersects
$\{x \in \RR^n \mymid e^Tx=0\}$ only in $0$, we conclude that a solution of Problem \eqref{maxmult}
yields the maximum multiplicity over $\Lambda_+ \setminus \{0\}$.} We also conclude by applying \new{Proposition
\ref{prop-component}}, that Problem
\eqref{maxmult} can be solved by computing at least one point per connected component of the real algebraic
sets $\new{\Secgamma_m}, m=1,\ldots,d-1$. This is a central
routine in computational algebraic geometry, for which exact algorithms have been designed,
see {\it e.g.}~\cite{basu1996combinatorial,RENEGAR1992255} or the monograph
\cite{BPR98} with its references. By exact representation, we mean via a rational univariate
representation \cite{Rouillier99}: this is a vector $(q,q_0,q_1,\ldots,q_{n}) \in \QQ[t]^{n+2}$
such that the set
\begin{equation}
\label{RUR}
\left\{ \left( \frac{q_1(t)}{q_0(t)}, \ldots, \frac{q_{n}(t)}{q_0(t)} \right) \mymid q(t)=0 \right\}
\end{equation}
intersects every connected component of the given algebraic set. In \eqref{RUR}, $q$ and $q_0$
are coprime, therefore the set is well defined and finite. Its cardinality is bounded above by
$\deg\,q$.

The best arithmetic complexity bounds for computing representations as in \eqref{RUR} are essentially polynomial in
the number of equations defining the algebraic set (which for $\new{\Secgamma_m}$ is at most $m\new{+1}$,
by Corollary \ref{eqGamma}) and in the maximum of their degree (bounded above by $d$) and singly
exponential in the number of variables. These come out of the so-called critical
points method \cite[Ch.\,16]{BPR98}, or the effective theory of polar varieties
\cite{SaSc03,SaSc04}.

In this paper we are not focusing on complexity results for hyperbolic programming, but our
efforts are especially devoted to the design of algebraic methods for this problem.
However, Proposition \ref{prop-component} and the mentioned results, according to the previous
complexity analysis, give a singly exponential algorithm to represent a solution to Problem
\ref{maxmult} in exact arithmetic, which is worth being highlighted.



We conclude this section with an example and comment on computational issues of Problem \ref{maxmult} for the case of
the elliptope. This is the feasible set of the SDP-relaxation of the MAX-CUT combinatorial
optimization problem \cite{goemans}.

\begin{example}[Elliptope]
\label{ellipt}
Let $d \in \NN$ and $n = \binom{d}{2}$, and let ${\scrE}_d$ be the $d$-elliptope. This is the spectrahedral cone of dimension
$n+1$ defined by the linear matrix inequality $A(x) \succeq 0$ with
\[
A(x)=
\left[
\begin{array}{cccc}
x_0    & x_{1,2} & \cdots & x_{1,d} \\
x_{1,2} &   x_0 & \ddots  & \vdots \\
\vdots & \ddots &   \ddots & x_{d-1,d} \\
x_{1,d} & \cdots &   x_{d-1,d} & x_0
\end{array}
\right].
\]
That is ${\scrE}_d = \{x=(x_0,x_{1,2},\ldots,x_{d-1,d}) \in \RR^{n+1} \mymid A(x) \succeq 0\}$.
It is the linear section of codimension $d-1$ of the cone of $d \times d$ real symmetric
positive semidefinite matrices whose main diagonal is constant. Every matrix in ${\scrE}_d$
is also called a correlation matrix (see \cite{LAURENT1995439} and references therein).
\new{Note that $e=(1,0,\ldots,0)$ (corresponding to the identity matrix $A(e) = \Id_d$) is in the
interior of ${\scrE}_d$ and that $L_e = \{x \in \RR^{n+1} \mymid x_0=1/d\}$.}
\end{example}

We propose two distinct tests on Example \ref{ellipt}.
The results that we present have been obtained on a desktop PC, with CPU architecture with the
following characteristics: {\tt Intel(R) Xeon(R) CPU E5-4620 0 @ 2.20GHz}. For the sake of
reproducibility, the corresponding
Maple scripts are made available on the webpage of the first
author\footnote{\url{www.unilim.fr/pages_perso/simone.naldi/software.html}}.

\begin{exper}
\label{exper:1}
\new{Since we work in the affine space $L_e$ defined in Example \ref{ellipt}, we} put $\new{x_0=1/d}$. We recall that
$\det A$ is hyperbolic with respect to $e = (1,0,\ldots,0)$.
The authors of \cite{LAURENT1995439} proved that the vertices of ${\scrE}_d$ (defined as those boundary
points whose normal cone is full-dimensional) are characterized as all the rank one matrices in $A(x)$
with entries in ${\pm 1}$. These are exactly the connected components \new{(in this case, isolated real points) of
$\Gamma_{d-1}$, that is $\Gamma_{d-1} \subset {\scrE}_d$, and each rank one matrix in $\Gamma_{d-1}$ maximizes the
multiplicity on ${\scrE}_d \setminus \{0\}$: indeed, $0$ is the unique positive semidefinite matrix of trace $0$,
hence the assumption in Remark \ref{remarknice} is satisfied. If $M$ is one of these rank one matrices, then $(1/d)M \in {\scrE}_d \cap
L_e$ maximizes the multiplicity on ${\scrE}_d \cap L_e$}.

These points can be computed efficiently in practice. We make use of the Maple library {\sc spectra}
\cite{spectra}, which is targeted to computing low rank solutions of linear matrix inequalities. As explained
in \cite{spectra}, the command {\tt SolveLMI(A,\{all\},[1])} (when called in a Maple worksheet where {\sc spectra}
has been previously loaded, and where the variable {\tt A} is instantiated to $A(x)$\new{, with $x_0=1/d$})
computes all components with highest multiplicity, namely isolated matrices of rank $1$.

To give an idea of performances, we are able to solve our problem for $d \leq 5$ in less than
half second, or for $d = 8$ (corresponding to an elliptope of dimension 21, and computing $2^7 = 128$ solutions)
in around 2.5 minutes. \hfill $\blacksquare$
\end{exper}

A second test is performed for the elliptope, without exploiting the spectrahedrality of ${\scrE}_d$,
but just relying on the hyperbolicity of $\det A$.

\begin{exper}
\label{exper:2}
We generate the sets $\Gamma_m$, $m=1, \ldots, d$, as the zero locus of the polynomials defined in Corollary \ref{eqGamma}.
For every $m$, we sample the real connected components of $\new{\Secgamma_m}$ and check how many of the solutions lie in
$\scrE_d$, using the Maple library {\sc raglib} \cite{raglib}. The function {\tt PointsPerComponents} allows us to sample the
real connected components of the sets $\new{\Secgamma_m}$, $m=1, \ldots, d$, computing rational parametrizations as in
\eqref{RUR}. The goal is to sample many points on the boundary of ${\scrE}_d$, possibly with different multiplicities. Note
that, \emph{a priori}, Proposition \ref{prop-component} guarantees that this method yields at least one feasible point for
the maximum multiplicity \new{in $\scrE_d \cap L_e$}.

The results are summarized in \new{Table \ref{tab:exp8}}.
\begin{table}[!ht]
\begin{tabular}{|cccccc|}
\hline
$d$ & $m$ & \# samples & $\#$ feasible & $\text{mult} = m$ & CPU time \\ 
\hline
$2$ & $1$ & 2 & 2 & 2 & 0.07 s \\ 
\hline
$3$ & $1$ & 10 & 4 & 0 & 0.7 s \\ 
$3$ & $2$ & 4 & 4 & 4 & 5.1 s \\ 
\hline
$4$ & $1$ & 36 & 10 & 6 & 26 s \\ 
$4$ & $2$ & 36 & 36 & 30 & 24 s \\ 
$4$ & $3$ & \new{8} & 8 & 8 & 34 s \\ 
\hline
\end{tabular}
\caption{Sample points on the $d-$elliptope}
\label{tab:exp8}
\end{table}
For a fixed $d$ and for a given multiplicity $m = 1, \ldots, d-1$, the number of sample points
computed by {\sc raglib} is given in the third column, those lying on ${\scrE}_d\new{\cap L_e}$ in the fourth,
and those of the expected multiplicity in the fifth column; then, we report on the average time
on 1000 tries on our standard desktop PC.

We also note that one computes points with multiplicity which can be larger than the expected
one (that is, matrices whose rank is smaller than expected). For example, for $d=3$
and $m=1$, the four points on $\new{\Secgamma_1}\cap \scrE_d$ actually belong to $\Gamma_2 \subset
\Gamma_1$, and correspond to those solutions computed in the subsequent step $m=2$. Moreover,
in contrast with {\sc spectra} in Test \ref{exper:1}, one can even sample multiplicities smaller
than the maximal one, as for the case of the $4$-elliptope. Already for the $5$-elliptope, however, the
computation becomes quite prohibitive, which is coherent with the exponential arithmetic
complexity \new{of the algorithms implemented in {\sc raglib}}. \hfill $\blacksquare$
\end{exper}

In these tests, we have seen that points of maximum multiplicity can be computed efficiently
in practice using LMI exact solvers in the case of hyperbolic polynomials with determinantal representation.
On the other hand, larger multiplicities can be computed on general hyperbolicity cones $\Lambda_+(f,e)$
by sampling the loci $\new{\Secgamma_m}$, but with clear limitations in terms of the degree of $f$.


\section{Hyperbolic programming}
\label{sec:HP}

Hyperbolic programming is a convex optimization problem specified as follows. We are given a homogeneous
polynomial $f \in \RR[x]$ of degree $d$, with $x=(x_1,\ldots,x_n)$, hyperbolic with respect to $e \in \RR^{n}$, and
a linear map $\ell \colon \RR^{n} \to \RR$.

The {\it hyperbolic program} associated to data $(f,e,\ell)$ is
\begin{equation}
\label{HP}
 \boxed{
\begin{aligned}
\ell^* \, = \, \inf & \,\,\, \ell(x) \\
 s.t. & \,\,\, x \in \new{\Seclambda_+(f,e)}
\end{aligned}
 }
\end{equation}


\new{{\bf Assumption.} We assume, without loss of generality, that in Problem \eqref{HP} $\ell(x)$ and $e^Tx$ are independent linear forms. Indeed,
if $\ell(x) = \lambda e^Tx$ for some $\lambda \in \RR$ and for all $x \in \RR^n$, then $\ell(x)$ is constant and equal to $\lambda$
over the feasible set $\new{\Seclambda_+(f,e)}$, and the hyperbolic program is trivial.} \\

Since the objective function of Problem \eqref{HP} is linear and the feasible set is convex, when the infimum is
attained at $x^*$, then $x^* \in \boundary(\new{\Seclambda_+}) \new{\subset} \boundary{\Lambda_+}$. The boundary $\boundary{\Lambda_+}$
is defined, locally, by the coefficients of the modified characteristic polynomial $\ch_{-x}(t)$, as in Corollary \ref{eqGamma}. We now
formalize this relationship between solutions to Problem \eqref{HP} and multiplicity loci.



\medskip
A local minimizer of a continuous function $\ell \colon \RR^{n} \to \RR$ on a set $S \subset \RR^{n}$ is a point
$x^* \in S$ such that there exists an open set $U \subset \RR^{n}$ with $x^* \in U$ and $\ell(x^*) \leq \ell(x)$
for all $x \in U \cap S$. We can reduce Problem \eqref{HP} to the computation of local
minimizers on the multiplicity loci $\new{\Secgamma_m}$ as follows.

\begin{theorem}
\label{theo:reduction}
Let $f,e,\ell$ be the data defining Problem \eqref{HP}. Let $x^* \in \new{\Seclambda_+(f,e)}$ with $\ell(x^*)=\ell^*$, and let
$m^* = \emph{mult}({x^*})$. Then $x^*$ is a local minimizer of $f$ on $\new{\Secgamma_{m^*}}$.
\end{theorem}
\begin{proof}
Since we are minimizing a linear function over a non-empty convex set, the minimizer (if it exists) belongs to the
boundary of the feasible set, hence $m^*>0$. 

Next, we denote by $\cc^*$ the connected component of $\new{\Secgamma_{m^*}}$ containing $x^*$. First, we suppose that 
$\cc^* \not\subset \Lambda_+$ and that $W \new{\cap L_e} \cap(\realset{f} \setminus \Lambda_+) \neq \emptyset$ holds for every open set
$W\subset\RR^{n}$ containing $x^*$. We show that this situation cannot occur. 
Indeed, since $\cc^* \not\subset \Lambda_+$, we easily deduce $W \cap  (\new{\Secgamma_{m^*} \setminus \Seclambda_+})\neq \emptyset$
for all $W$ as above. Let $B_k$ be the open ball with center $x^*$
and radius $1/k$, for $k \in \NN \setminus \{0\}$. For all such $k$, we can choose $x(k) \in B_k\cap(\new{\Secgamma_{m^*}\setminus
\Seclambda_+})$, yielding a sequence $x(k) \xrightarrow[]{k \rightarrow
  \infty} x^*$. We deduce that $\lambda_1(x(k))<0$ holds for all $k$,
and hence $\lambda_{m^*+1}(x(k)) \leq 0$ (because at least $m^*$
eigenvalues of $x(k)$ must vanish). Passing to the limit, we find $\lambda_{m^*+1}(x^*) = 0$; since $x^* \in \new{\Seclambda_+}$, we get
$\lambda_j(x^*)=0$ for $j=1, \ldots, m^*+1$, which implies $\mult({x^*}) \geq m^*+1$, which is a contradiction.

Now, two cases remain to be analyzed:

{\it Case A} : $\cc^* \subset \Lambda_+$. Thus $\ell(x^*) \leq  \ell(x)$ for all $x \in \cc^*$, that is $x^*$ is a (global)
minimizer of $\ell$ on $\cc^*$, hence a local minimizer of $\ell$ on $\new{\Secgamma_{m^*}}$.

{\it Case B} : There is an open set $W \subset \RR^{n}$ with $x^* \in W$ and $W \cap (\realset{f} \setminus \Lambda_+)\new{\cap L_e}=
\emptyset$. In other words, $W$ meets $\realset{f}$ only at feasible points. Hence $x^*$ minimizes $\ell$ on $W \cap \new{\Secgamma_{m^*}}$,
hence it is a local minimizer of $\ell$ on $\new{\Secgamma_{m^*}}$.
\end{proof}

We now give a formal description of an \emph{algorithm} for Problem \eqref{HP}. The idea is to represent local minimizers
of the map $\ell$ on the set $\new{\Secgamma_m}$ via first-order conditions involving additional Lagrange multipliers. Let
$f_1,\ldots,f_m$ be the polynomials defining $\Gamma_m$ (see Corollary \ref{eqGamma}).
A local minimizer $x^*$ of $\ell$ on $\Gamma_m$ is encoded by the following system of equations:
\begin{equation}
\label{LagSyst}
\begin{aligned}
f_1 = 0, \,\,\, \ldots, \,\,\, f_{m} = 0, \new{e^Tx=1} \\
z_1 \nabla f_1 + \cdots + z_{m} \nabla f_m \new{+ z_{m+1}e} & = \nabla \ell 
\end{aligned}
\end{equation}
which means that there exists $z^* \in \RR^{m\new{+1}}$ such that $(x^*,z^*)$ satisfies system \eqref{LagSyst}. When the number
of singular points of the complex set $\{x \in \CC^n \mymid \forall\,i=1, \, f_i(x) = 0, \new{e^Tx=1}\}$ is finite, which turns out
to be often satisfied, the solutions of system \eqref{LagSyst} consist of these singular points and the smooth
minimizers of $\ell$. \new{Note that if $\ell(x)$ and $e^Tx$ are dependent linear forms, then $e$ and $\nabla\ell$ are multiples,
and hence system \eqref{LagSyst} has infinitely many solutions (all feasible points are critical). Our assumption that $\ell,e^Tx$ are
independent excludes this pathological situation}.

We suppose now that we are given a routine {\sf RP} that, when given
as input a zero-dimensional ideal $I \subset \RR[x,z]$, 
returns the rational parametrization \eqref{RUR} for the finite set $\zeroset{I \cap \RR[x]} \subset \CC^{n}$. Algorithms to compute such
parametrizations have appeared {\it e.g.}~in \cite{Rouillier99,GiLeSa01,SymbHom}. The routine {\sf LAG} is supposed to build the
system \eqref{LagSyst} from data $f,e,\ell,m$. The formal description of our algorithm is as follows:

\begin{algorithm}
\caption{\bf SolveHP}
\label{solvehp}
\begin{algorithmic}[1]
\Procedure{SolveHP}{$f,e,\ell$}
\State $L \leftarrow \{\,\}$
\For{$m=1,\ldots,d-1$}
\State $L \leftarrow L \cup {\sf RP}({\sf LAG}(f,e,\ell,m))$
\EndFor
\State \Return{$L$}
\EndProcedure
\end{algorithmic}
\end{algorithm}

The output of {\bf SolveHP} is a set of rational parametrizations. The union of their solutions contains the solution
to Problem \eqref{HP}, according to Theorem \ref{theo:reduction}.

\subsection{Examples}

We have implemented Algorithm \ref{solvehp} in Maple and used it to recover exact information on
examples from the literature.

\begin{example}[Quartic symmetroids]
We consider the list of nodal quartic symmetroids given in \cite{Ottem2015}. The authors of
\cite{Ottem2015} associate to every transversal quartic spectrahedron
$$
\spec = \{x = (x_0,x_1,x_2,x_3) \in \RR^4 \mymid A(x) \succeq 0\},
$$
\new{where $A(x)=x_0A_0+x_1A_1+x_2A_2+x_3A_3$, for some $A_i \in \ss_4(\RR)$},
a couple $(\rho,\sigma)$ of nonnegative integers, where $\rho$ corresponds to the number of
nodes (quadratic singularities) of the \new{real} projective hypersurface $\{x \in \PP^3(\RR) \mymid \det A(x)\new{=0}\}$, and
$\sigma \leq \rho$ is the number of nodes lying on $\partial \spec$. We denote the quartic
of type $(\rho,\sigma)$ in the list in \cite{Ottem2015} by $\spec_{\rho,\sigma}$.

The goal is to perform a random analysis on the solutions of SDP instances over these sets. While this
is similar to \cite[Table 2]{stu}, our exact viewpoint can certify the multiplicity at a
given solution; indeed, once the representation \eqref{RUR} is computed, isolating the
real roots of $q$ allows to compute the signs of the coefficients of the characteristic
polynomial $\det (t \Id-A(x))$ exactly, hence to decide feasibility and multiplicity.
The same is not possible with standard SDP solvers.
\begin{table}[!ht]
\begin{tabular}{|c|rr||c|rr|}
\hline
$\spec_{\rho,\sigma}$ & $m^*=1$ & $m^*=2$ & $\spec_{\rho,\sigma}$ & $m^*=1$ & $m^*=2$ \\
\hline
$\spec_{2,2}$ & { 98}\,\% & { 2}\,\% & $\spec_{4,0}$ & { 100}\,\% & { 0}\,\% \\
$\spec_{4,4}$ & { 62}\,\% & { 38}\,\% & $\spec_{6,2}$ & { 32}\,\% & { 68}\,\% \\
$\spec_{6,6}$ & { 58}\,\% & { 42}\,\% & $\spec_{8,4}$ & { 17}\,\% &  { 83}\,\% \\
$\spec_{8,8}$ & { 22}\,\% & { 78}\,\% & $\spec_{10,6}$ & { 7}\,\% &  { 93}\,\% \\
$\spec_{10,10}$ & { 75}\,\% & { 25}\,\% & $\spec_{6,0}$ & { 100}\,\% &  { 0}\,\% \\
$\spec_{2,0}$ & { 100}\,\% & { 0}\,\% & $\spec_{8,2}$ & { 15}\,\%  & { 85}\,\% \\
$\spec_{4,2}$ & { 36}\,\% & { 64}\,\% & $\spec_{10,4}$ & { 14}\,\% & { 86}\,\% \\
$\spec_{6,4}$ & { 46}\,\% & { 54}\,\% & $\spec_{8,0}$ & { 100}\,\% & { 0}\,\%  \\
$\spec_{8,6}$ & { 63}\,\% & { 37}\,\% & $\spec_{10,2}$ & { 18}\,\% & { 82}\,\%  \\
$\spec_{10,8}$ & { 86}\,\% & { 14}\,\% & $\spec_{10,0}$ & { 100}\,\% & { 0}\,\% \\
\hline
\end{tabular}
\caption{Multiplicities on random quartic symmetroids}
\label{tab:quarticspectra}
\end{table}
We draw random linear forms $\ell \in \QQ[x_0,x_1,x_2,x_3]_1$ with coefficients uniformely
distributed in $\ZZ \cap [-100,100]$, and we compute the solution in \eqref{HP} with $f =
\det A$ and $e=\Id_4$. \new{Note that in this case the standard section of the hyperbolicity
cone is given by $1=\left\langle \Id_4, A\right\rangle=\text{Trace}(A)$, hence we restrict the homogeneous pencil
$A(x)$ to the affine space of matrices with trace $1$.}

The multiplicity of a solution $x^* \in \spec_{\rho,\sigma}$
in this case corresponds to the corank of $A(x^*)$.
In Table \ref{tab:quarticspectra} we report on the percentage for the multiplicity
at a minimizer on 1000 tries. There are only two possible multiplicities, that is $1$ and $2$.
Feasible points with multiplicity $2$ correspond to the singularities of the determinant
lying on $\spec_{\rho,\sigma}$.

We finally generated other representatives of the classes. We observe that percentages can
change; indeed these depend not only on the topology of the symmetroid, but also on how the
singularities on $\partial \spec_{\rho, \sigma}$ are exposed. We believe that this approach can
be useful to solve similar classification problems of larger size. \hfill $\blacksquare$
\end{example}

Algorithm \ref{solvehp} still works without the assumption that $f$ has a determinantal representation.
We test our algorithm on one such example.

\begin{example}
Let $A(x)$ be a $5 \times 5$ homogeneous symmetric linear matrix in $4$ variables $x_0,x_1,x_2,x_3$,
with $A(e) \succ 0$ for some $e \in \RR^4$. Let $f = \det A$. Then the directional derivative of $f$
in direction $e$, that is the polynomial
\[
D_e^{(1)}(f) = \sum_{i=0}^n e_i \frac{\partial f}{\partial x_i},
\]
is hyperbolic with respect to $e$ (hence the same is true for the $k-$th derivative $D_e^{(1)}(f)$,
$1 \leq k \leq 5$, by induction), but in general does not admit a determinantal representation.
For example, let
\[
A(x) = 
{\footnotesize
\begin{bmatrix}
x_0+x_3 & 2x_1+2x_3 & x_1+3x_3 & x_2 & x_2+3x_3 \\
2x_1+2x_3 & x_0+4x_1+3x_3 & x_1-x_2+6x_3 & x_1+x_2-2x_3 & x_1+x_2+4x_3 \\ 
x_1+3x_3 & x_1-x_2+6x_3 & x_0+x_1+8x_3 & -x_1-x_2-3x_3 & -x_1-x_2+6x_3 \\
x_2 & x_1+x_2-2x_3 & -x_1-x_2-3x_3 & x_0+x_2+x_3 & x_1+2x_2-x_3 \\
x_2+3x_3 & x_1+x_2+4x_3 & -x_1-x_2+6x_3 & x_1+2x_2-x_3 & x_0+x_2+4x_3 \\
\end{bmatrix}.
}
\]
Then $f=\det A$ is hyperbolic with respect to $e = (1,0,0,0)$ (corresponding to
the identity matrix $\Id_5$), and the quintic real hypersurface $\{x \in \RR^4 \mymid f(x) = 0\}$ has
four singularities. \new{As in the previous example, we cut the hyperbolicity cone with the
condition $\text{Trace}(A(x))=1$ defining the affine space $L_e$.} The derivative $D_e^{(1)}(f)$ defines a singular
quartic, still with four nodes and hyperbolic with respect to $e$ (see Figure \ref{Fig:NodalQuartic}),
which is not representable as a determinant of a symmetric pencil. Let
$\Lambda_+(D_e^{(1)}(f),e)$ be its hyperbolicity cone.

Optimizing generic linear functions over $\new{\Seclambda_+(D_e^{(1)}(f),e)}$ yields solutions of multiplicity one
(smooth boundary points) for 64\% of the time, and solutions corresponding to singular points (of multiplicity
2) for 36\% of the time, on average. An example of multiplicity two is any multiple of the vector with coordinates
\[
\begin{array}{llll}
x_0 = \frac{1}{2} \,\,\,\, & x_1 = 0 \,\,\,\, & x_2 = \frac{1}{2} \,\,\,\, & x_3 = 0
\end{array}
\]
which in this case are rational numbers. A smooth point on the boundary of $\Lambda_+(D_e^{(1)}(f),e)$ (multiplicity one),
whose coordinates are given as elements of certified rational intervals, with 10 significant decimal digits, is:
\[
\begin{array}[b]{l}
x_0 \in [\frac{1697352983372573029285}{1180591620717411303424}, \frac{212169122921574321063}{147573952589676412928}] \approx 1.437713900  \\[.3em]
x_1 \in [-\frac{29707767148026666593}{147573952589676412928}, -\frac{29707767148024593931}{147573952589676412928}] \approx -0.2013076605 \\[.3em]
x_2 \in [-\frac{18765770300641154993}{73786976294838206464}, -\frac{18765770300640685591}{73786976294838206464}] \approx -0.2543236116 \\[.3em]
x_3 \in [\frac{21153099339285995043}{1180591620717411303424}, \frac{661034354352767111}{36893488147419103232}] \approx 0.01791737208. \\[.3em]
\end{array}\eqno{\blacksquare}
\]
\end{example}

In our last example, we show how our algorithm can certify lower
bounds of Renegar's method, which uses derivative
cones for hyperbolic programming.

\begin{example}[Nie, Parrilo, Sturmfels \cite{Nie2008}; Saunderson, Parrilo \cite{Saunderson2015}]
We consider the semidefinite representation of the $3$-ellipse, as computed in \cite{Nie2008}.
Given $n$ points $p_1, \ldots, p_n$ in $\RR^2$, and a nonnegative real number $D$, the $n$-ellipse
is the plane compact curve consisting of those points the sum of whose distances to $p_1, \ldots,
p_n$ is $D$ (which is called the radius of the ellipse). This set is the boundary of a spectrahedral
hyperbolicity cone ${\mathcal E}_n$, for every
$D$. Moreover, one has the stronger property that the algebraic boundary of ${\mathcal E}_n$ is a
determinantal hypersurface \cite{Nie2008}. However, it may be
advantageous not to compute such a determinantal representation and
work with the hyperbolic polynomial directly, \new{which in the case of the $3-$ellipse ${\mathcal E}_3$ with foci $(0,0),(3,0)$ and $(0,4)$ and radius 8, is the degree-$8$-polynomial
{\small
\begin{align*}
f & = 9x^8-72x^7z+36x^6y^2-96x^6yz-1564x^6z^2-216x^5y^2z+960x^5yz^2+9912x^5z^3+ \\ 
& +54x^4y^4-288x^4y^3z-4748x^4y^2z^2+12256x^4yz^3+70782x^4z^4-216x^3y^4z+ \\
& +1920x^3y^3z^2+17424x^3y^2z^3-71040x^3yz^4-262296x^3z^5+36x^2y^6-288x^2y^5z-\\
& -4804x^2y^4z^2+27712x^2y^3z^3+137228x^2y^2z^4-564384x^2yz^5-616140x^2z^6-\\
& -72xy^6z+960xy^5z^2+7512xy^4z^3-76416xy^3z^4-389688xy^2z^5+1372608xyz^6+\\
& +1610280xz^7+9y^8-96y^7z-1620y^6z^2+15456y^5z^3+58014y^4z^4-349728y^3z^5-\\
& -457380y^2z^6+1723680yz^7+893025z^8.
\end{align*}
}
restricted to the plane $z=1$.
}
Let $A=A(x,y,z)$ be the linear matrix representation of ${\mathcal E}_3$
given in \cite[Example 1]{Saunderson2015}, and hence $f = \det A$. The corresponding $3$-ellipse ${\mathcal E}_3$ has
the semidefinite representation $\{(x,y) \in \RR^2 \mymid A(x,y,1) \succeq 0\}$.
The boundary $\partial{\mathcal E}_n$ could contain one or more singularity, and if this happens, these coincide with
some of the base points $p_1, \ldots, p_k$. \new{This is the case for the $3-$ellipse we consider, which contains
the point $(3,0)$. The polynomial $f$ is hyperbolic with respect to $e=(1,1,1)$, hence $L_e$ is given by the equation
$x+y+z=1$}.

Our exact algorithm for hyperbolic programs, as we have already remarked, is able to certify rational intevals containing the coordinates of a solution, its multiplicity, and
also the optimal value of the linear function on the solution. In order to measure the error when considering Renegar relaxations for solving hyperbolicity programs, we consider
the relaxations of ${\mathcal E}_3$, namely the hyperbolicity cones of the derivatives of $f$ in the direction $e=(1,1,1)$. The infimum of the linear function
$\ell(x,y,z)=x+2y+3z+4$ on ${\mathcal E}_3\new{\cap L_e}$ is attained at the unique point of multiplicity 2, that is \new{at $(3/4,0,1/4)$ (projectively equivalent to $(3,0,1)$)}.
We optimize the same linear function over the derivative relaxations and look at the sequence of optimal values.

\begin{table}[!ht]
\begin{tabular}{|cccccc|}
\hline
$k$ & $\approx x^*$ & $m^*$ & $\ell(x^*)$ & Degree of $q(t)$ & \new{Alg. deg. of $x^*$} \\
\hline
0 & $(0.750, 0.000, 0.250)$ & 2 & 5.500000000 & 56 & \new{1} \\
1 & $(0.759, -0.018, 0.258)$ & 1 & 5.499158216 & 42 & \new{30} \\
2 & $(0.797, -0.051, 0.250)$ & 1 & 5.456196445 & 30 & \new{26} \\
3 & $(0.862, -0.116, 0.254)$ & 1 & 5.392044926 & 20 & \new{20} \\
4 & $(0.981, -0.254, 0.273)$ & 1 & 5.292250029 & 12 & \new{12} \\
5 & $(1.336, -0.762, 0.426)$ & 1 & 5.090555573 & 6 & \new{6} \\
\hline
\end{tabular}
\caption{Derivative relaxations of the $3-$ellipse}
\label{tab:3ellipse}
\end{table}

In \new{Table \ref{tab:3ellipse}}, $k$ denotes the order of derivation of $f$, and $x^*,m^*,f(x^*)$ denote the minimizer, its multiplicity, \new{and the optimal
value of $\ell$ on the given derivative cone $\Lambda_+(D^{(k)}_e(f),e)$}, respectively \new{(here $D^{(k)}_e(f)$ denotes the $k-$th derivative of $f$ in direction $e$).
Moreover, we report in the fifth and sixth column, the degree of the polynomial $q(t)$ in the
rational representation which is computed by our algorithm ({\it cf. \eqref{RUR}}) and the  degree of the coordinates of $x^*$ (as algebraic numbers over $\QQ$).
We first remark that the value in the fifth column decreases when considering derivative relaxations, which implies the following fact: the higher the derivative
relaxation order is, the faster the exact representation, and hence the lower bound, can be computed. We also remark that for $k=0,1,2$ the value in the sixth column is
lower. This is because the polynomial $q$ is reducible, and for $k=0$ (resp $k=1,2$) has a linear (resp. degree 30, degree 26) factor which corresponds
to the minimum polynomial of the extensions $\QQ[x^*_i]$ over $\QQ$. Indeed, in these cases the variety of critical points factors over $\QQ$.}


\end{example}

\end{document}